\documentclass[preprint,12pt]{elsarticle}
\usepackage{amsthm,amsmath,amssymb}
\usepackage[colorlinks=true,citecolor=black,linkcolor=black,urlcolor=blue]{hyperref}
\usepackage{graphicx}
\usepackage{float}
\usepackage{epstopdf}
\usepackage{epsfig}
\usepackage{extarrows,chngpage,array,float,eqnarray}

\usepackage{longtable}
\usepackage{tabulary}
\usepackage{graphicx}
\usepackage{subfigure}
\usepackage{epstopdf}
\usepackage{setspace}
\usepackage[space]{grffile}
\usepackage{latexsym}
\usepackage{textcomp}
\usepackage{longtable}
\usepackage{tabulary}
\usepackage{booktabs,array,multirow}
\usepackage[table]{xcolor}
\usepackage{amssymb,amsmath,latexsym}
\usepackage[export]{adjustbox}
\usepackage{color}
\usepackage{threeparttable}
\usepackage{diagbox}
\usepackage{bm}
\usepackage{mathrsfs}

\theoremstyle{plain}
\newtheorem{theorem}{Theorem}[section]
\newtheorem{lemma}[theorem]{Lemma}

\newtheorem{proposition}[theorem]{Proposition}

\theoremstyle{definition}
\newtheorem{definition}[theorem]{Definition}

\journal{Linear Algebra and its Applications}
\begin{document}

\title{\textbf{Two spectral extremal results for graphs with given order and rank}}
\author{
	\small Xiuqing Li,\ \ Xian'an Jin\footnote{Corresponding author},\ \ Chao Shi,\ \ Ruiling Zheng\\[0.2cm]
	\small School of Mathematical Sciences, Xiamen University,\\
	\small Xiamen, Fujian 361005, P. R. China\\[0.2cm]
	\small E-mails: xiuqingli2021@163.com; xajin@xmu.edu.cn;\\
	                \quad \quad \ cshi@aliyun.com; rlzheng2017@163.com}
\date{}

\begin{abstract}
The spectral radius and rank of a graph are defined to be the spectral radius and rank of its adjacency matrix, respectively. It is an important problem in spectral extremal graph theory to determine the extremal graph that has the maximum or minimum spectral radius over certain families of graphs. Monsalve and Rada [Extremal spectral radius of graphs with rank 4, Linear Algebra Appl. 609 (2021) 1–11] obtained the extremal graphs with maximum and minimum spectral radii among all graphs with order $n$ and rank $4$. In this paper, we first determine the extremal graph which attains the maximum spectral radius among all graphs with any given order $n$ and rank $r$, and further determine the extremal graph which attains the minimum spectral radius among all graphs with order $n$ and rank $5$.

\vskip0.2cm

\noindent{\bf Keywords:} Rank of graphs; Extremal graphs; Maximum spectral radius; Minimum spectral radius

\vskip0.2cm

\end{abstract}

\maketitle

\section{Introduction}
Graphs considered in the paper are all simple, connected and undirected. Let $G=(V(G),E(G))$ be a graph. For $v \in V(G)$, the degree $d(v)$ is the cardinality of the neighborhood $N_{G}(v)$ (or $N(v)$ for short) of $v$ in $G$. Let $A(G)$ be the adjacency matrix of $G$. The characteristic polynomial of a graph $G$ is the determinantal expansion of $xI-A(G)$, denoted by $\phi(G,x)$. According to the famous Perron-Frobenius theorem, the largest eigenvalue $\rho(G)$ of $A(G)$ is exactly the spectral radius of $G$ and there is a unique positive unit eigenvector corresponding to $\rho(G)$, called the principal eigenvector of $G$.

Let $G$ be a graph with vertex set $V(G)=\{v_{1},v_{2},\dots,v_{k}\}$ and $\textbf{m}=(n_{1},n_{2},\dots,n_{k})$ be a vector of positive integers. Denote by $G \circ \textbf{m}$, the graph obtained from $G$ by replacing each vertex $v_{i}$ with an independent set $V_{i}$ with $n_{i}$ vertices $v_{i}^{1}, v_{i}^{2}, \dots, v_{i}^{n_{i}}$ and joining each vertex in $V_{i}$ with each vertex in $V_{j}$ if and only if $v_{i}v_{j} \in E(G)$. The resulting graph $G \circ \textbf{m}$ is said to be obtained from $G$ by multiplication of vertices by Chang, Huang and Yeh in \cite{01}. Further, let $G$ be a graph of order $k$, we define $M_{n}(G)$ to be the set of all graphs $G \circ (n_{1},n_{2},\dots,n_{k})$ with $\sum_{i=1}^{k}n_{i}=n$. Moreover, for a given set of graphs $\{H_{1},\dots,H_{l}\}$, we denote the set $\bigcup_{i=1}^{l}M_{n}(H_{i})$ by $M_{n}(H_{1},\dots,H_{l})$.

Let $G$ be a connected graph of order $n$ and $R(G)$ be its rank. Sciriha \cite{04} proved that $R(G)=i$ if and only if $G \in M_{n}(K_{i})$ for $i=2,3$, where $K_i$ is the complete graph of order $i$. Chang, Huang and Yeh \cite{01,05} characterized the set of all connected graphs with rank 4 and 5, respectively. They obtained the set of connected graphs of order $n$ and rank 5 is $$M_{n}(G_{1}, G_{2}, \dots, G_{24}),$$
where the graphs $G_{1}, G_{2}, \dots, G_{24}$ are shown in Figure \ref{Fig.1.1}.

\begin{figure}[H]
	\centering
	\includegraphics[width=12cm]{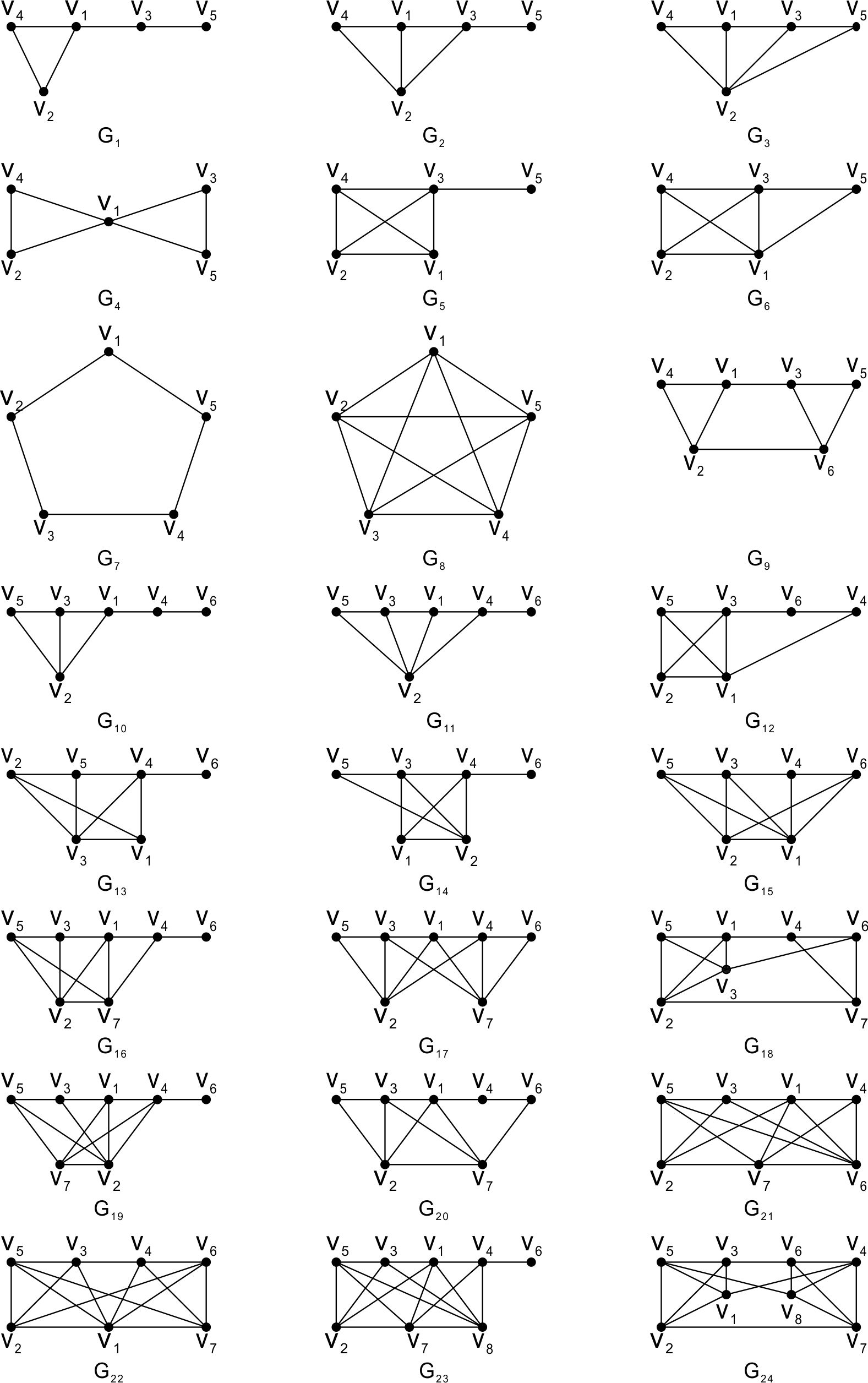}
	\caption{Reduced graphs of rank 5.}
	\label{Fig.1.1}
\end{figure}

For a given class of graphs $\mathscr{G}$, there are many results on characterizing the extramal graphs with maximum and minimum spectral radius among $M_{n}(\mathscr{G})$. For example, in \cite{06}, Stevanovi\'{c}, Gutman and Rehman determined the extremal graphs with the maximum and minimum spectral radii in $M_{n}(K_{p})$. Monsalve and Rada \cite{07} obtained the extremal graphs with maximum and minimum spectral radii among all connected graphs of order $n$ and rank $4$. In the same article, they conjectured that in $M_{n}(P_{k})$, $P_{k}\circ (1,\dots,1,\lfloor \frac{n-k+2}{2} \rfloor,\lceil \frac{n-k+2}{2} \rceil,1,\dots,1)$ and $P_{k}\circ (\lfloor \frac{n-k+2}{2} \rfloor,1,\dots,1,$ $\lceil \frac{n-k+2}{2} \rceil)$ attain the maximum and minimum spectral radius, respectively, and $C_{k}\circ (\lfloor \frac{n-k+2}{2} \rfloor,\lceil \frac{n-k+2}{2} \rceil,1,\dots,1)$ attains the maximum spectral radius in $M_{n}(C_{k})$. Recently, Lou, Zhai \cite{02} and Sun, Das \cite{03} independently proved the above conjectures on the extremal graphs with the maximum spectral radius in $M_{n}(P_{k})$ and $M_{n}(C_{k})$ by using different techniques, and they independently constructed a class of graphs disproving the conjecture on the minimum spectral radius in $M_{n}(P_{k})$.

The Tur\'{a}n graph $T(n,r)$ is the complete $r$-partite graph on $n$ vertices where its part sizes are as equal as possible. In this paper, we first determine the extremal graph that attains the maximum spectral radius with any given order and rank, and obtain:

\begin{theorem}\label{1}
$T(n,r)$ is the unique extremal graph that attains the maximum spectral radius among all graphs of order $n$ and rank $r$.
\end{theorem}

However, it seems that it is a difficult task to find the extremal graph that attains the minimum spectral radius with given order and rank. In this paper, we focus on graphs with order $n$ and rank $5$, and obtain:

\begin{theorem}\label{2}
The extremal graph that attains the minimum spectral radius among all connected graphs of order $n$ and rank 5 is:
	\begin{itemize}
		\item $G_{7}=C_{5}$, for $n=5$;
		\item $G_{1} \circ(1,1,1,1,n-4)$, for $6\leq n \leq 10$;
		\item $G_{10} \circ(1,1,1,1,1,n-5)$, for $n=11$;
		\item $G_{10} \circ (1,1,1,1,k,n-k-4)$, where $k=\lfloor \frac{6n-37-\sqrt{24n+1}}{18} \rfloor$ or $\lceil \frac{6n-37-\sqrt{24n+1}}{18} \rceil$, for $n \geq 12$.
	\end{itemize}
\end{theorem}

\section{The proof of Theorem \ref{1}}

We will use the following results to prove Theorem \ref{1}.

\begin{theorem}\cite{01} \label{2.1}
	Suppose that $G$ and $H$ are two graphs. If $H \in M_{n}(G)$, then $R(H)=R(G)$.
\end{theorem}

\begin{theorem}\cite{08} \label{2.2}
	Let $T(n,r)$ be the $r$-partite Tur\'{a}n graph of order n. If $G$ is a $K_{r+1}$-free graph of order $n$, then $\rho(G)<\rho(T(n,r))$ unless $G=T(n,r)$.
\end{theorem}

\begin{proof}[\textbf{Proof of Theorem \ref{1}}]
	Let $G$ be a graph of order $n$ and rank $r$. We claim that $G$ is a $K_{r+1}$-free graph. Otherwise, since $K_{r+1}$ is a subgraph of $G$, selecting the rows and columns corresponding to the vertices in $K_{r+1}$ can obtain a nonzero minor of order $r+1$ of $A(G)$, i.e.,
	\begin{align} \det \left(\begin{array}{cccc}
	0 &1&\cdots &1 \\
	1 &0&\cdots &1 \\
	\vdots &\vdots &\ddots &\vdots  \\
	1&1 &\cdots &0 \\
	\end{array}\right)_{(r+1)\times(r+1)}=(-1)^{r}\cdot r \neq 0.\notag
	\end{align}
	Therefore, we have $R(G)\geq r+1$, a contradiction. Since $T(n,r)=K_r \circ (\lceil \frac{n}{r} \rceil,\dots,\lceil \frac{n}{r} \rceil,\lfloor \frac{n}{r} \rfloor,\dots,\lfloor \frac{n}{r} \rfloor) \in M_{n}(K_r)$, by Theorem \ref{2.1}, we have $R(T(n,r))=R(K_r)=r$. By Theorem \ref{2.2}, we obtain $\rho(G)<\rho(T(n,r))$ unless $G=T(n,r)$.
\end{proof}

\section{The proof of Theorem \ref{2}}
In this section, we focus on the extremal graph that has the minimum spectral radius among all connected graphs of order $n$ and rank 5. We firstly outline our proof for Theorem \ref{2}.

{\bf Step 1.} We first apply a result of Monsalve and Rada in \cite{07} to prove that the extremal graph with minimum spectral radius belongs to $M_{n}(G_{1},G_{7}$, $G_{10})$. 

{\bf Step 2.} Then, using the method of comparing characteristic polynomials, we characterize the extremal graph with minimum spectral radius in $M_{n}(G_{1})$, $M_{n}(G_{7})$ and $M_{n}(G_{10})$, respectively. 

{\bf Step 3.} Next, for $n\geq 12$, we compare the spectral radii of these three types of extremal graphs by some well-known results and obtain that the extremal graph of order $n$ and rank $5$ with minimum spectral radius is $G_{10}\circ(1,1,1,1,k,n-4-k)$ for some integer $k$. Further, we determine $k \in \{ \lfloor \frac{6n-37-\sqrt{24n+1}}{18} \rfloor, \lceil \frac{6n-37-\sqrt{24n+1}}{18} \rceil \}$. 

{\bf Step 4.} Finally, for $5\leq n\leq 11$, we obtain the extremal graphs by calculating directly the spectral radii of the extremal graphs in $M_{n}(G_{1})$, $M_{n}(G_{7})$ and $M_{n}(G_{10})$, respectively.

\subsection{Step 1}
We begin with recalling a well-known result.

\begin{theorem}\cite{09} \label{3.1}
	If $H$ is a proper subgraph of a connected graph $G$, then $\rho(H) <\rho(G)$.
\end{theorem}

In \cite{07}, Theorem \ref{3.1} is used to prove the following results.

\begin{theorem}\cite{07} \label{3.2}
	Let $G$ be a connected graph with $k$ vertices and $\textbf{m}=(n_{1},n_{2},\dots,n_{k})$ a vector of positive integers. If $v_{1}v_{2} \in E(G)$, then $$\rho((G-v_{1}v_{2}) \circ \textbf{m}) < \rho(G \circ \textbf{m}).$$
\end{theorem}

\begin{theorem}\cite{07} \label{3.3}
	Let $G$ be a connected graph with $k$ vertices and $\textbf{m}=(n_{1},n_{2},\dots,n_{k})$ a vector of positive integers. If $v_{i}v_{j} \notin E(G)$ and $N(v_{i}) \subsetneq N(v_{j})$, then
	$$\rho(G \circ (n_{1},\dots,n_{i},\dots,n_{j},\dots,n_{k})) < \rho(G \circ (n_{1},\dots,n_{i}-1,\dots,n_{j}+1,\dots,n_{k})).$$
\end{theorem}

By Theorem \ref{3.2}, we obtain the following proposition.

\begin{proposition}\label{sets}
	Let $G$ be the extremal graph with minimum spectral radius among all connected graphs of order $n$ and rank 5. Then $G \in M_{n}(G_{1}, G_{7}, G_{10})$.
\end{proposition}

\begin{proof}
	Let $\textbf{m}_{1}=(n_{1}, n_{2}, n_{3}, n_{4}, n_{5}),$ $\textbf{m}_{2}=(n_{1}, n_{2}, n_{3}, n_{4}, n_{5}, n_{6}),$ $\textbf{m}_{3}=(n_{1}, n_{2},$ $n_{3}, n_{4}, n_{5}, n_{6}, n_{7})$ and $\textbf{m}_{4}=(n_{1}, n_{2}, n_{3},$ $n_{4}, n_{5}, n_{6}, n_{7}, n_{8})$ be arbitrary vectors of positive integers with $\sum_{i=1}n_{i}=n$.
   	As a consequence of Theorem \ref{3.2}, we have
   	\begin{align*}
	&\rho(G_{1} \circ \textbf{m}_{1}) < \rho(G_{i} \circ \textbf{m}_{1}), i=2,3,4,5,6,8,\\
	&\rho(G_{10} \circ \textbf{m}_{2}) < \rho(G_{j} \circ \textbf{m}_{2}), j=11,12,13,14,15.\\
	\end{align*}
	
	Thus,
	$$G \in M_{n}(G_{1}, G_{7}, G_{9}, G_{10}, G_{16}, G_{17}, G_{18}, G_{19}, G_{20},G_{21},G_{22}, G_{23}, G_{24}).$$
	
	Let $H_{1}=G_{1} \circ (1, 1, 1, 1, 2)$, $H_{2}=G_{10} \circ (1, 1, 1, 1, 1, 2)$, $H_{3}=G_{10} \circ (1, 1, 1, 1, 2, 1)$ and $H_{4}=G_{10} \circ (1, 1, 1, 1, 1, 3)$, as shown in Figure \ref{Fig.3.1}.
	
	Obiviously,
	\begin{itemize}
	\item $H_{1}$ is the spanning proper subgraph of $G_{9}$;
	\item $H_{2}$ is the spanning proper subgraph of $G_{i},i \in \{16,17,18,19,21,22\}$;
	\item $H_{3}$ is the spanning proper subgraph of $G_{20}$;
	\item $H_{4}$ is the spanning proper subgraph of $G_{j},j \in \{23,24\}$.
	\end{itemize}

	Therefore, it follows from Theorem \ref{3.2} that
	\begin{align*}
	&\rho(G_{1} \circ \textbf{m}_{2}')=\rho(H_{1} \circ \textbf{m}_{2})<\rho(G_{9} \circ \textbf{m}_{2}),\\
	&\rho(G_{10} \circ \textbf{m}_{3}')=\rho(H_{2} \circ \textbf{m}_{3})<\rho(G_{i} \circ \textbf{m}_{3}), i=16,17,18,19,21,22,\\
	&\rho(G_{10} \circ \textbf{m}_{3}'')=\rho(H_{3} \circ \textbf{m}_{3})<\rho(G_{20} \circ \textbf{m}_{3}),\\
	&\rho(G_{10} \circ \textbf{m}_{4}')=\rho(H_{4} \circ \textbf{m}_{4})<\rho(G_{j} \circ \textbf{m}_{4}), j=23, 24,
	\end{align*}
	where $\textbf{m}_{2}'=(n_{1}, n_{2}, n_{3}, n_{4}, n_{5}+n_{6}),$ $\textbf{m}_{3}'=(n_{1}, n_{2}, n_{3}, n_{4}, n_{5}, n_{6}+n_{7}),$ $\textbf{m}_{3}''=(n_{1}, n_{2}, n_{3}, n_{4},n_{5}+n_{7},n_{6})$ and $\textbf{m}_{4}'=(n_{1}, n_{2}, n_{3}, n_{4}, n_{5}, n_{6}+n_{7}+n_{8}).$
	
	Hence, $G \in M_{n}(G_{1}, G_{7}, G_{10})$.
	
	\begin{figure}[H]
		\centering
		\includegraphics{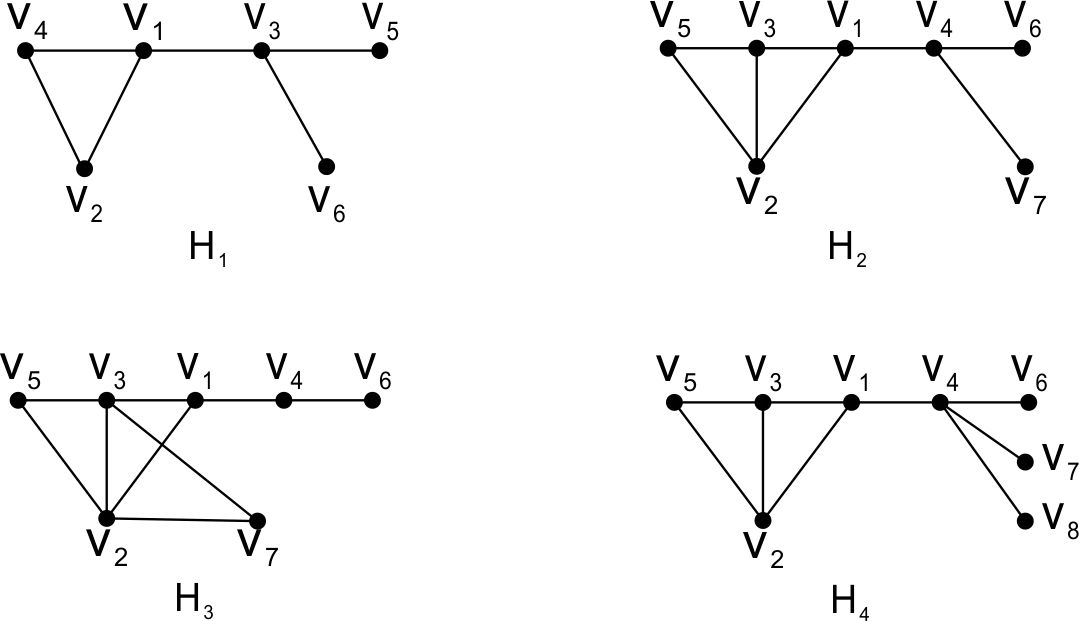}
		\caption{The graphs $H_{i}, i=1,2,3,4$.}
		\label{Fig.3.1}
	\end{figure}

\end{proof}

\subsection{Step 2}
In this subsection we characterize the extremal graphs with minimum spectral radii in $M_{n}(G_{1}), M_{n}(G_{7})$ and $M_{n}(G_{10})$, respectively. To accomplish this, let's introduce some classic results in spectral graph theory.

\begin{definition}\cite{10} \label{3.4}
	Let $A$ be an $n \times n$ real matrix whose rows and columns are indexed by $X=\{1,2,\dots,n\}$. We partition $X$ into ${X_{1},X_{2},\dots,X_{k}}$ in order and rewrite $A$ according to the partition ${X_{1},X_{2},\dots,X_{k}}$ as follows:
	\begin{align} A=\left(\begin{array}{ccc}
	A_{1,1} &\cdots &A_{1,k} \\
	\vdots &\ddots &\vdots  \\
	A_{k,1} &\cdots &A_{k,k} \\
	\end{array}\right),\notag
	\end{align}
	where $A_{i,j}$ is the block of $A$ formed by rows in $X_{i}$ and the columns in $X_{j}$. Let $b_{i,j}$ denote the average row sum of $A_{i,j}$. Then the matrix $B=[b_{i,j}]$ will be called the \textbf{quotient matrix} of the partition of $A$. In particular, when the row sum of each block $A_{i,j}$ is constant, the partition is called an \textbf{equitable partition}.
\end{definition}

\begin{theorem}\cite{10} \label{3.5}
	Let $A \geq 0$ be an irreducible square matrix, $B$ be the quotient matrix of an equitable partition of $A$. Then the spectrum of $A$ contains the spectrum of $B$ and $\rho(A)=\rho(B).$
\end{theorem}

\begin{theorem}\cite{11} \label{3.6}
	Let $G$ and $H$ be two connected graphs such that $\phi(H,x)>\phi(G,x)$ for $x\geq \rho(G)$. Then $\rho(H)<\rho(G)$.
\end{theorem}

\begin{theorem}\cite{09} \label{cm}
	Let $K_{n_{1},n_{2},\dots,n_{k}}$ be the complete multipartite graph of order $n$. Then
	$$\phi(K_{n_{1},n_{2},\dots,n_{k}},x)=x^{n-k}(1-\sum_{i=1}^{k}\frac{n_{i}}{x+n_{i}})\prod_{i=1}^{k}(x+n_{i}).$$
\end{theorem}

The following Propositions \ref{g1}, \ref{g10} and \ref{g7} give the extremal graph which attains the minimum spectral radius in $M_{n}(G_{1})$, $M_{n}(G_{10})$ and $M_{n}(G_{7})$, respectively.

\begin{proposition}\label{g1}
	The extremal graph in $M_{n}(G_{1})$ which attains minimum spectral radius is of the form
	$$G_{1} \circ (1, 1, 1, k, n-k-3),$$
	where $1 \leq k\leq \frac{n-3}{2}.$
\end{proposition}

\begin{proof}
	Since $N(v_{5})=\{v_{3}\} \subsetneq N(v_{1})$ and $v_{1}v_{5} \notin E(G_{1})$, then by Theorem \ref{3.3} we have
	$$\rho(G_{1} \circ (1,n_{2},n_{3},n_{4},n_{5}+n_{1}-1)) \leq \rho(G_{1} \circ (n_{1},n_{2},n_{3},n_{4},n_{5})),$$
	with equality if and only if $n_{1}=1$. It follows that the extremal graph in $M_{n}(G_{1})$ which attains minimum spectral radius is of the form $F=G_{1} \circ (1,n_{2},n_{3},n_{4},n_{5})$.
	
	Then $V(F)$ can be naturally partitioned into $5$ parts:
	$$\{V_{1},V_{2},V_{3},V_{4},V_{5}\},$$
	where $V_{i}=\{v_{i}^{1}, \dots, v_{i}^{n_{i}}\}, i=1,2,3,4,5$. Obviously, this partition of $A(F)$ is equitable and the corresponding quotient matrix B is
	\begin{align} B=\left(\begin{array}{ccccc}
	0 &n_{2} &n_{3} &n_{4} &0 \\
	1 &0 &0 &n_{4} &0 \\
	1 &0 &0 &0 &n_{5} \\
	1 &n_{2} &0 &0 &0 \\
	0 &0 &n_{3} &0 &0\\
	\end{array}\right).\notag
	\end{align}
	
	Then the characteristic polynomial of the quotient matrix $B$ is:
	\begin{equation}
	\begin{split}
	\phi(B,x)&=x^{5}-(n_{2}+n_{3}+n_{4}+n_{2}n_{4}+n_{3}n_{5})x^{3}-2n_{2}n_{4}x^{2}+\\
	&(n_{2}n_{3}n_{4}+n_{2}n_{3}n_{5}+n_{3}n_{4}n_{5}+n_{2}n_{3}n_{4}n_{5})x+2n_{2}n_{3}n_{4}n_{5}.
	\end{split}\notag
	\end{equation}
	
	Since $R(A(F))=5$, by Theorem \ref{3.5} we have $\phi(F,x)=x^{n-5} \phi(B,x)$ and $\rho(F)=\rho(A(F))=\rho(B)$.
	
	Note that $G_{1} \circ (1,n_{2},n_{3},n_{4},n_{5}) \cong G_{1} \circ (1,n_{4},n_{3},n_{2},n_{5})$. Therefore, without loss of generality, we suppose that $n_{4} \geq n_{2}.$
	
	\textbf{Claim 1.} $n_{2}=1.$
	
	Assume $n_{2} \geq 2$, let $F_{1}=G_{1} \circ (1,n_{2}-1,n_{3},n_{4}+1,n_{5})$ then
	\begin{equation}
	\begin{split}
	r(x)&=\phi(F_{1},x)-\phi(F,x)\\
	&=x^{n-5}(n_{4}-n_{2}+1)(x^{3}+2x^{2}-(n_{3}+n_{3}n_{5})x-2n_{3}n_{5})\\
	&=x^{n-5}(n_{4}-n_{2}+1)\left(x(x^{2}-n_{3}(n_{5}+1))+2(x^{2}-n_{3}n_{5})\right).
	\end{split} \notag
	\end{equation}
	
	Since $n_{4} \geq n_{2}$, we have $n_{4}-n_{2}+1 >0$. It is clear that $K_{n_{3},n_{5}+1}$ is a proper subgraph of $F$, we obtain $\rho(F)>\rho(K_{n_{3},n_{5}+1})=\sqrt{n_{3}(n_{5}+1)}$, then $r(x)>0$ for $x \geq \rho(F)$.
	
	Thus, by Theorem \ref{3.6}, we have $\rho(F_{1})<\rho(F)$ which contradicts to the extremality of $F$.
	
    \textbf{Claim 2.} $n_{3}=1.$

    Now $F=G_{1} \circ (1,1,n_{3},n_{4},n_{5})$, we claim that $n_{5} \geq n_{3}$. If not, let $F_{2}=G_{1} \circ (1,1,n_{5},n_{4},n_{3})$, then
    \begin{equation}
    \begin{split}
    r(x)&=\phi(F_{2},x)-\phi(F,x)=x^{n-4}(x^{2}-n_{4})(n_{3}-n_{5}).
    \end{split}\notag
    \end{equation}

    Since $n_{3}>n_{5}$, we have $n_{3}-n_{5} >0$. It can be seen that $K_{n_{4},2}$ is a proper subgraph of $F$, we obtain $\rho(F)>\rho(K_{n_{4},2})=\sqrt{2n_{4}}$, then $r(x)>0$ for $x \geq \rho(F)$.

    Thus, by Theorem \ref{3.6}, we have $\rho(F_{2})<\rho(F)$, a contradiction. Therefore $n_{5} \geq n_{3}$.

    Next, we assume $n_{3} \geq 2$, let $F_{3}=G_{1} \circ (1,1,n_{3}-1,n_{4},n_{5}+1)$ then
    \begin{equation}
    \begin{split}
    r(x)&=\phi(F_{3},x)-\phi(F,x)\\
    &=x^{n-5}\left((n_{5}-n_{3}+1)(x^{3}-(2n_{4}+1)x-2n_{4})+x(x^{2}-n_{4})\right).
    \end{split}\notag
    \end{equation}

    Since $n_{5} \geq n_{3}$, we have $n_{5}-n_{3}+1>0$. It is clear that $K_{n_{4},1,1}$ is a proper subgraph of $F$, by Theorem \ref{cm}, we obtain $\rho(F)>\rho(K_{n_{4},1,1})=(\sqrt{8n_{4}+1}+1)/2$, then $r(x)>0$ for $x \geq \rho(F)$.

    Thus, by Theorem \ref{3.6}, we have $\rho(F_{3})<\rho(F)$, which contradicts to the extremality of $F$.

    \textbf{Claim 3.} $n_{5} \geq n_{4}.$

    Now $F=G_{1} \circ (1,1,1,n_{4},n_{5})$. Otherwise, let $F_{4}=G_{1} \circ (1,1,1,n_{5},n_{4})$ then
    \begin{equation}
    \begin{split}
    r(x)&=\phi(F_{4},x)-\phi(F,x)=x^{n-3}(x+2)(n_{4}-n_{5}).
    \end{split}\notag
    \end{equation}

    Since $n_{4} > n_{5}$ and $\rho(F)>0$, then $r(x)>0$ for $x \geq \rho(F)$. By Theorem \ref{3.6}, we have $\rho(F_{4})<\rho(F)$ which contradicts to the extremality of $F$, thus $n_{5} \geq n_{4}$.

    From above three claims, we conclude that the extremal graph with minimum spectral radius in $M_{n}(G_{1})$ is of the form $G_{1} \circ (1, 1, 1, k, n-k-3)$, where $1 \leq k\leq (n-3)/2$.
\end{proof}

Similarly, we characterize the extremal graph with minimum spectral radius in $M_{n}(G_{10})$.

\begin{proposition}\label{g10}
	The extremal graph in $M_{n}(G_{10})$ which attains minimum spectral radius is of the form
	$$G_{10} \circ (1, 1, 1, 1, k, n-k-4),$$
	where $1 \leq k\leq \frac{n-4}{2}.$
\end{proposition}

\begin{proof}
	By Theorem \ref{3.3}, we have
	$$\rho(G_{10} \circ (1,n_{2},n_{3},n_{4},n_{5},n_{6}+n_{1}-1)) \leq \rho(G_{10} \circ (n_{1},n_{2},n_{3},n_{4},n_{5},n_{6})),$$
	with equality if and only if $n_{1}=1$. Thus, we may suppose that the extremal graph in  $M_{n}(G_{10})$ which attains minimum spectral radius is of the form $F=G_{10} \circ (1,n_{2},n_{3},n_{4},n_{5},n_{6})$.
	
   Similarly, we obtain
	\begin{align} B=\left(\begin{array}{cccccc}
	0 &n_{2} &n_{3} &n_{4} &0 &0 \\
	1 &0 &n_{3} &0 &n_{5} &0 \\
	1 &n_{2} &0 &0 &n_{5} &0 \\
	1 &0 &0 &0 &0 &n_{6} \\
	0 &n_{2} &n_{3} &0 &0 &0 \\
	0 &0 &0 &n_{4} &0 &0\\
	\end{array}\right),\notag
	\end{align}
	is the quotient matrix of an equitable partition of $A(F)$. The characteristic polynomial of the quotient matrix $B$ is:
	\begin{equation}
	\begin{split}
	\phi(B,x)&=x(x^{5}-(n_{2}+n_{3}+n_{4}+n_{2}n_{3}+n_{2}n_{5}+n_{3}n_{5}+n_{4}n_{6})x^{3}\\
	&-(2n_{2}n_{3}+2n_{2}n_{3}n_{5})x^{2}+(n_{2}n_{3}n_{4}+n_{2}n_{4}n_{5}+n_{2}n_{4}n_{6}\\
	&+n_{3}n_{4}n_{5}+n_{3}n_{4}n_{6}+n_{2}n_{3}n_{4}n_{6}+n_{2}n_{4}n_{5}n_{6}+n_{3}n_{4}n_{5}n_{6})x\\
	&+2n_{2}n_{3}n_{4}n_{5}+2n_{2}n_{3}n_{4}n_{6}+2n_{2}n_{3}n_{4}n_{5}n_{6}).
	\end{split}\notag
	\end{equation}
	
	Since $R(A(F))=5$, by Theorem \ref{3.5} we have $\phi(F,x)=x^{n-6} \phi(B,x)$ and $\rho(F)=\rho(A(F))=\rho(B)$.
	
	Note that $G_{10} \circ (1,n_{2},n_{3},n_{4},n_{5},n_{6}) \cong G_{10} \circ (1,n_{3},n_{2},n_{4},n_{5},n_{6})$. Therefore, without loss of generality, we suppose that $n_{3} \geq n_{2}.$
	
	\textbf{Claim 1.} $n_{2}=1.$
	
	Assume $n_{2} \geq 2$, let $F_{1}=G_{10} \circ (1,n_{2}-1,n_{3}+1,n_{4},n_{5},n_{6})$ then
	\begin{equation}
	\begin{split}
	r(x)&=\phi(F_{1},x)-\phi(F,x)\\
	&=x^{n-5}(n_{3}-n_{2}+1)(x^{3}+2(1+n_{5})x^{2}-(n_{4}+n_{4}n_{6})x-2n_{4}n_{5}\\
	&-2n_{4}n_{6}-2n_{4}n_{5}n_{6})\\
	&=x^{n-5}(n_{3}-n_{2}+1)(x(x^{2}-n_{4}(n_{6}+1))+2n_{5}(x^{2}-n_{4}(n_{6}+1))\\
	&+2(x^{2}-n_{4}n_{6})).
	\end{split}\notag
	\end{equation}
	
	Since $n_{3} \geq n_{2}$, we have $n_{3}-n_{2}+1 >0$. It is clear that $K_{n_{4},n_{6}+1}$ is a proper subgraph of $F$, we obtain $\rho(F)>\rho(K_{n_{4},n_{6}+1})=\sqrt{n_{4}(n_{6}+1)}$, then $r(x)>0$ for $x \geq \rho(F)$.
	
	Thus, by Theorem \ref{3.6}, we have $\rho(F_{1})<\rho(F)$ which contradicts to the extremality of $F$.
		
	\textbf{Claim 2.} $n_{3}=1.$
	
	Now $F=G_{10} \circ (1,1,n_{3},n_{4},n_{5},n_{6})$, we claim that $n_{5} \geq n_{3}$. If not, let $F_{2}=G_{10} \circ (1,1,n_{5},n_{4},n_{3},n_{6})$, then
	\begin{equation}
	\begin{split}
	r(x)&=\phi(F_{2},x)-\phi(F,x)=x^{n-5}(n_{3}-n_{5})(x^{2}-n_{4}n_{6})(x+2).
	\end{split}\notag
	\end{equation}
	
	Since $n_{3} > n_{5}$, we have $n_{3}-n_{5} >0$. It can be seen that $K_{n_{4},n_{6}}$ is a proper subgraph of $F$, we obtain $\rho(F)>\rho(K_{n_{4},n_{6}})=\sqrt{n_{4}n_{6}}$, then $r(x)>0$ for $x \geq \rho(F)$.
	
	Thus, by Theorem \ref{3.6}, we have $\rho(F_{2})<\rho(F)$, a contradiction. Therefore $n_{5} \geq n_{3}$.
	
	Next, we assume $n_{3} \geq 2$, let $F_{3}=G_{10} \circ (1,1,n_{3}-1,n_{4},n_{5}+1,n_{6})$ then
	\begin{equation}
	\begin{split}
	r(x)&=\phi(F_{3},x)-\phi(F,x)\\
	&=x^{n-5}(x+2)\left((n_{5}-n_{3}+2)(x^{2}-n_{4}(n_{6}+1))+n_{4}\right).
	\end{split}\notag
	\end{equation}
	
	Since $n_{5} \geq n_{3}$, we have $n_{5}-n_{3}+2>0$. It is clear that  $K_{n_{4},n_{6}+1}$ is a proper subgraph of $F$, we obtain $\rho(F)>\rho(K_{n_{4},n_{6}+1})=\sqrt{n_{4}(n_{6}+1)}$, then $r(x)>0$ for $x \geq \rho(F)$.
	
	Thus, by Theorem \ref{3.6}, we have $\rho(F_{3})<\rho(F)$ which contradicts to the extremality of $F$.
	
	\textbf{Claim 3.} $n_{4}=1.$
	
	Now $F=G_{10} \circ (1,1,1,n_{4},n_{5},n_{6})$, we claim that $n_{6} \geq n_{4}$. If not, let $F_{4}=G_{10} \circ (1,1,1,n_{6},n_{5},n_{4})$, then
	\begin{equation}
	\begin{split}
	r(x)&=\phi(F_{4},x)-\phi(F,x)=x^{n-5}(n_{4}-n_{6})(x^{2}-x-2n_{5})(x+1).
	\end{split}\notag
	\end{equation}
	
	Since $n_{4} > n_{6}$, we have $n_{4}-n_{6} >0$. It can be seen that $K_{n_{5},1,1}$ is a proper subgraph of $F$, we obtain $\rho(F)>\rho(K_{n_{5},1,1})=(\sqrt{8n_{5}+1}+1)/2$, then $r(x)>0$ for $x \geq \rho(F)$.
	
	Thus, by Theorem \ref{3.6}, we have $\rho(F_{4})<\rho(F)$, a contradiction. Therefore $n_{6} \geq n_{4}$.
	
	Next, we assume $n_{4} \geq 2$, let $F_{5}=G_{10} \circ (1,1,1,n_{4}-1,n_{5},n_{6}+1)$ then
	\begin{equation}
	\begin{split}
	r(x)&=\phi(F_{5},x)-\phi(F,x)\\
	&=x^{n-5}(x+1)\left((n_{6}-n_{4}+2)(x^{2}-x-2n_{5}-2)+2\right).
	\end{split}\notag
	\end{equation}
	
	Since $n_{6} \geq n_{4}$, we have $n_{6}-n_{4}+2>0$. It is clear that $H \circ(n_{5},1,1,1)$ is a proper subgraph of $F$, where $H$ is shown in Figure \ref{Fig.3.2}, we obtain $\rho(F)>\rho(H \circ(n_{5},1,1,1))=(\sqrt{8n_{5}+9}+1)/2$, then $r(x)>0$ for $x \geq \rho(F)$.
	
	Thus, by Theorem \ref{3.6}, we have $\rho(F_{5})<\rho(F)$, which contradicts to the extremality of $F$.
	
	\begin{figure}[H]
		\centering
		\includegraphics{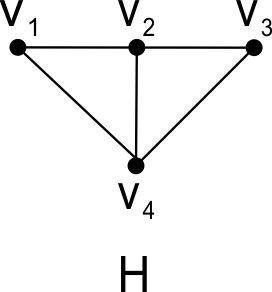}
		\caption{The graph $H$.}
		\label{Fig.3.2}
	\end{figure}
	
	\textbf{Claim 4.} $n_{6} \geq n_{5}.$
	
	Now $F=G_{10} \circ (1,1,1,1,n_{5},n_{6})$. Otherwise, let $F_{6}=G_{10} \circ (1,1,1,1,n_{6},n_{5})$ then
	\begin{equation}
	\begin{split}
	r(x)&=\phi(F_{6},x)-\phi(F,x)=x^{n-4}(x+1)^{2}(n_{5}-n_{6}).
	\end{split}\notag
	\end{equation}
	
	Since $n_{5}>n_{6}$ and $\rho(F)>0$, then $r(x)>0$ for $x \geq \rho(F)$, by Theorem \ref{3.6}, we have $\rho(F_{6})<\rho(F)$ which contradicts to the extremality of $F$, thus $n_{6} \geq n_{5}$.
	
	It follows from above four claims that the extremal graph with minimum spectral radius in $M_{n}(G_{10})$ is of the form $G_{10} \circ (1, 1, 1, 1, k, n-k-4),$ where $1 \leq k\leq (n-4)/2$.
\end{proof}

Next we determine the extremal graph with minimum spectral radius in $M_{n}(G_{7})$.

\begin{proposition}\label{g7}
	The extremal graph in $M_{n}(G_{7})$ which attains minimum spectral radius is
	$$G_{7} \circ (\lceil \frac{n-3}{2} \rceil, 1, \lfloor \frac{n-3}{2} \rfloor, 1,1).$$
\end{proposition}

\begin{proof}
	Suppose that the extremal graph in $M_{n}(G_{7})$ which attains minimum spectral radius is of the form $F=G_{7} \circ (n_{1},n_{2},n_{3},n_{4},n_{5})$. Similarlly, we obtain
	\begin{align} B=\left(\begin{array}{ccccc}
	0 &n_{2} &0 &0 &n_{5} \\
	n_{1} &0 &n_{3} &0 &0  \\
	0 &n_{2} &0 &n_{4} &0 \\
	0 &0 &n_{3} &0 &n_{5} \\
	n_{1} &0 &0 &n_{4} &0\\
	\end{array}\right),\notag
	\end{align}
	is the quotient matrix of an equitable partition of $A(F)$ and the characteristic polynomial of $B$ is:
	\begin{align*}
	\phi(B,x)&=x^{5}-(n_{1}n_{2}+n_{2}n_{3}+n_{1}n_{5}+n_{3}n_{4}+n_{4}n_{5})x^{3}+(n_{1}n_{2}n_{3}n_{4}+\\
	&n_{1}n_{2}n_{3}n_{5}+n_{1}n_{2}n_{4}n_{5}+n_{1}n_{3}n_{4}n_{5}+n_{2}n_{3}n_{4}n_{5})x-2n_{1}n_{2}n_{3}n_{4}n_{5}.
	\end{align*}
	
	Since $R(A(F))=5$, by Theorem \ref{3.5} we have $\phi(F,x)=x^{n-5} \phi(B,x)$ and $\rho(F)=\rho(A(F))=\rho(B)$.
	
	Without loss of generality, we may suppose that $n_{1}=\text{max}\{n_{i},i=1,2,3,4,5\}$ and $n_{2} \leq n_{5}$, then we have the following claims.
	
	\textbf{Claim 1.} $n_{2} \leq n_{3}$ and $n_{5} \leq n_{4}$.
		
	Suppose that $n_{2}>n_{3}$. Let $F_{1}=G_{7} \circ (n_{1},n_{3},n_{2},n_{4},n_{5})$, then
	\begin{equation*}
	r(x)=\phi(F_{1},x)-\phi(F,x)=x^{n-2}(n_{1}-n_{4})(n_{2}-n_{3}).
	\end{equation*}
		
	Since $n_{1}=\text{max}\{n_{i},i=1,2,3,4,5\}$, we have $n_{1}\geq n_{4}$. And if $n_{1}=n_{4}$, then $F_{1} \cong F$. Thus, without loss of generality, we may suppose that $n_{1}>n_{4}$.
		
	Since $n_{2}>n_{3}$, $n_{1}>n_{4}$ and $\rho(F)>0$, then $r(x)>0$ for $x \geq \rho(F)$. By Theorem \ref{3.6}, we have $\rho(F_{1})<\rho(F)$ which contradicts to the extremality of $F$.
		
	Similarly, we obtain $n_{5} \leq n_{4}$.
		
	\textbf{Claim 2.} $n_{4} \leq n_{3}$.
	
	Suppose to the contrary that $n_{4}>n_{3}$. Let $F_{2}=G_{7} \circ (n_{1},n_{2},n_{4},n_{3},n_{5})$, then
	\begin{equation*}
	r(x)=\phi(F_{2},x)-\phi(F,x)=x^{n-2}(n_{2}-n_{5})(n_{3}-n_{4}).
	\end{equation*}
	
	Since $n_{5}\geq n_{2}$ and if $n_{5}=n_{2}$, then $F_{2} \cong F$. Thus without loss of generality we may suppose that $n_{5}>n_{2}$.
	
	Since $n_{4}>n_{3}$, $n_{5}>n_{2}$ and $\rho(F)>0$. Then $r(x)>0$ for $x \geq \rho(F)$. By Theorem \ref{3.6}, we have $\rho(F_{2})<\rho(F)$ which contradicts to the extremality of $F$.

    From above two claims, we have $n_{1}\geq n_{3}\geq n_{4}\geq n_{5}\geq n_{2}$. Next, we will prove $n_{2}=n_{4}=n_{5}=1$ and $n_{1}-n_{3}\leq 1$.

	\textbf{Claim 3.} $n_{2}=n_{5}$
	
	Assume $n_{2}<n_{5}$, let $F_{3}=G_{7} \circ (n_{1}+n_{5}-n_{2},n_{2},n_{3},n_{4},n_{2})$ then
	\begin{align*}
	r(x)&=\phi(F_{3},x)-\phi(F,x)\\
	&=x^{n-5}(n_{5}-n_{2})((n_{1}-2n_{2}+n_{4})x^{3}-(n_{1}-n_{2})((n_{3}n_{4}+n_{2}n_{3}+n_{2}n_{4})x\\
	&-2n_{2}n_{3}n_{4})) \\
	&\geq x^{n-5}(n_{5}-n_{2})(n_{1}-n_{2})\left(x^{3}-(n_{3}n_{4}+n_{2}n_{3}+n_{2}n_{4})x+2n_{2}n_{3}n_{4}\right) \\
	&=x^{n-5}(n_{5}-n_{2})(n_{1}-n_{2})g(x).
	\end{align*}
	
	Since $n_{1}\geq n_{3} \geq n_{4} \geq n_{5}>n_{2}\geq1$, we have $n_{1}-n_{2}>0$ and $n_{5}-n_{2}>0$. It is clear that $K_{n_{3},n_{2}+n_{4}}$ is a proper subgraph of $F$, we obtain $\rho(F)>\rho(K_{n_{3},n_{2}+n_{4}})=\sqrt{n_{3}(n_{2}+n_{4})}$.
	
	Since $g(\sqrt{n_{3}(n_{2}+n_{4})})>0$ and $\sqrt{n_{3}(n_{2}+n_{4})}>\sqrt{(n_{3}(n_{2}+n_{4})+n_{2}n_{4})/3}$, where $\sqrt{\left(n_{3}(n_{2}+n_{4})+n_{2}n_{4}\right)/3}$ is the largest root of $g'(x)$, we have $r(x)>0$ for $x \geq \rho(F)$.
	
	Thus, by Theorem \ref{3.6}, we have $\rho(F_{3})<\rho(F)$ which contradicts to the extremality of $F$.
	
	Note that $G_{7} \circ (n_{1},n_{2},n_{3},n_{4},n_{5}) \cong G_{7} \circ (n_{1},n_{2},n_{4},n_{3},n_{5})$ when $n_{2}=n_{5}$, therefore without loss of generality we may suppose that $n_{3} \geq n_{4}.$
	
	\textbf{Claim 4.} $n_{4}=1.$
	
	Assume $n_{4}\geq 2$, let $F_{4}=G_{7} \circ (n_{1},n_{2},n_{3}+1,n_{4}-1,n_{5})$ then
	\begin{equation}
	\begin{split}
	r(x)&=\phi(F_{4},x)-\phi(F,x)\\
	&=x^{n-5}(x-n_{5})(n_{3}-n_{4}+1)(x^{2}+n_{5}x-2n_{1}n_{5}).
	\end{split}\notag
	\end{equation}
	
	Since $n_{1} \geq n_{3}\geq n_{4}\geq n_{5}=n_{2}$, we have $n_{3}-n_{4}+1>0$. It can be seen that $K_{n_{1},2n_{5}}$ is a proper subgraph of $F$, we obtain $\rho(F)>\rho(K_{n_{1},2n_{5}})=\sqrt{2n_{1}n_{5}}>n_{5}$, then $r(x)>0$ for $x \geq \rho(F)$.
	
	Thus, by Theorem \ref{3.6}, we have $\rho(F_{4})<\rho(F)$ which contradicts to the extremality of $F$, therefore $n_{4}=1$ and hence $n_{2}=n_{5}=1$.
	
	\textbf{Claim 5.} $n_{1}-n_{3}\leq 1.$
	
	Now $F=G_{7} \circ (n_{1},1,n_{3},1,1)$. Assume $n_{1}\geq n_{3}+2$, let $F_{5}=G_{7} \circ (n_{1}-1,1,n_{3}+1,1,1)$ then
	\begin{equation}
	\begin{split}
	r(x)&=\phi(F_{5},x)-\phi(F,x)=x^{n-5}(3x-2)(n_{1}-n_{3}-1).
	\end{split}\notag
	\end{equation}
	
	Since $n_{1} \geq n_{3}+2$, we have $n_{1}-n_{3}-1>0$. It is clear that $K_{n_{1},2}$ is a proper subgraph of $F$, we obtain $\rho(F)>\rho(K_{n_{1},2})=\sqrt{2n_{1}}>1$, then $r(x)>0$ for $x \geq \rho(F)$.
	
	Thus, by Theorem \ref{3.6}, we have $\rho(F_{5})<\rho(F)$ which contradicts to the extremality of $F$, therefore $n_{1}-n_{3}\leq 1$ and hence $n_{1}=\lceil(n-3)/2 \rceil, n_{3}=\lfloor (n-3)/2 \rfloor.$
	
	From above five claims, we obtain $G_{7} \circ (\lceil \frac{n-3}{2} \rceil, 1, \lfloor \frac{n-3}{2} \rfloor, 1,1)$ attains the minimum spectral radius in $M_{n}(G_{7})$.
\end{proof}

\subsection{Step 3}

We first prove that the extremal graph with minimum spectral radius in $M_{n}(G_{1},G_{7})$ must be in $M_{n}(G_{1})$ by the following lemma.

\begin{lemma}\label{3.7}
	For $n \geq 8$, we have $\rho(G_{1} \circ (1,1,1,\lfloor \frac{n-3}{2} \rfloor, \lceil \frac{n-3}{2} \rceil )) < \rho(G_{7} \circ (\lceil \frac{n-3}{2} \rceil, 1, \lfloor \frac{n-3}{2} \rfloor, 1,1))$.
\end{lemma}

\begin{proof}
Let $F_{1}=G_{1} \circ (1,1,1,\lfloor \frac{n-3}{2} \rfloor, \lceil \frac{n-3}{2} \rceil )$ and $F_{2}=G_{7} \circ (\lceil \frac{n-3}{2} \rceil, 1, \lfloor \frac{n-3}{2} \rfloor, 1,1)$.

For $8\leq n \leq 12$, we use the MATLAB software to calculate the spectral radii of $F_{i}$ for $i=1,2$, as shown in the Table \ref{tab1}.

  \begin{table}[H]\tiny
	\centering
	\caption{$\rho(F_{i})$.}
	\label{tab1}
	\resizebox{\textwidth}{!}{
		\begin{tabular}{ccc}
			\hline
			\ \ \ \ \ \ \ \ \ \ $n$ \ \ \ \ \ \ \ \ \ \ & \ \ \ \ \ \ \ \ \ \ $\rho(F_{1})$ \ \ \ \ \ \ \ \ \ \ &  \ \ \ \ \ \ \ \ \ \ $\rho(F_{2})$ \ \ \ \ \ \ \ \ \ \ \\ \hline
			8 &  2.7676  &  2.9764 \\
			9 & 3.1474 &  3.2176   \\
			10 &  3.1713  & 3.4630   \\
			11 & 3.5047   &  3.6737  \\
			12 &  3.5223  & 3.8879  \\ \hline
			\end{tabular}}
\end{table}

So let us assume that $n\geq 13$.
	
	\textbf{Case 1.} $n-3=2k$ is even.
	
	In this case, $F_{1}=G_{1} \circ (1,1,1,k,k)$ and $F_{2}=G_{7} \circ (k,1,k,1,1)$, then
	\begin{equation}
	\begin{split}
	r(x)&=\phi(F_{1},x)-\phi(F_{2},x)=x^{n-5}\left((k-1)x^{3}-2kx^{2}-k^{2}x+4k^{2}\right).
	\end{split}\notag
	\end{equation}
	
	It can be seen that $K_{2k,1}$ is a proper subgraph of $F_{2}$, we obtain $\rho(F_{2})>\rho(K_{2k,1})=\sqrt{2k}$. Since $n \geq 13$, we have $r(\sqrt{2k})>0$ and $\sqrt{2k}>(2k+k\sqrt{3k+1})/3(k-1)$. Since $(2k+k\sqrt{3k+1})/3(k-1)$ is the largest root of $r'(x)$, we obtain $r(x)>0$ for $x \geq \rho(F_{2})$.
	
	Thus by Theorem \ref{3.6}, we have $\rho(F_{1})<\rho(F_{2})$.
	
	\textbf{Case 2.} $n-3=2k+1$ is odd.
	
	In this case, $F_{1}=G_{1} \circ (1,1,1,k,k+1)$ and $F_{2}=G_{7} \circ (k+1,1,k,1,1)$, then
	\begin{equation}
	\begin{split}
	r(x)&=\phi(F_{1},x)-\phi(F_{2},x)=x^{n-5}k\left(x^{3}-2x^{2}-(k+1)x+4k+\right).
	\end{split}\notag
	\end{equation}
	
	It is clear that $K_{2k+1,1}$ is a proper subgraph of $F_{2}$, we obtain $\rho(F_{2})>\rho(K_{2k+1,1})=\sqrt{2k+1}$. Since $n \geq 13$, we have $r(\sqrt{2k+1})>0$ and $\sqrt{2k+1}>(2+\sqrt{3k+7})/3$. Since $(2+\sqrt{3k+7})/3$ is the largest root of $r'(x)$, we obtain $r(x)>0$ for $x \geq \rho(F_{2})$.
	
    Thus by Theorem \ref{3.6}, we have $\rho(F_{1})<\rho(F_{2})$.
		
\end{proof}

Next we prove the extremal graph with minimum spectral radius in $M_{n}(G_{1},G_{10})$ must be in $M_{n}(G_{10})$. We need the following theorem.

\begin{theorem}\cite{12} \label{3.8}
	Let $G$ be a graph with m edges and n vertices. Then $\rho(G) \leq \sqrt{2m-n+1}$, with equality if and only if $G$ is isomorphic to the star $K_{1,n-1}$ or the  complete graph $K_{n}$.
\end{theorem}

\begin{lemma} \label{3.9}
  Let $G_{1} \circ(1, 1, 1,k,n-k-3)$ be the extremal graph with minimum spectral radius in $M_{n}(G_{1})$ for $n \geq 12$. Then $2\leq k \leq \frac{n-3}{2}$.
\end{lemma}

\begin{proof}
    We denote $F_k=G_{1} \circ(1, 1, 1,k,n-k-3)$ for convenience. By Proposition \ref{g1}, we have $1\leq k \leq (n-3)/2$.

    For $12 \leq n \leq 18$, we use the MATLAB software to calculate the spectral radii of $F_{k}$, as shown in the Table \ref{tab2}, where the minimum spectral radius is bolded.
    \begin{table}[H]\scriptsize
    	\centering
    	\caption{$\rho(F_k)$.}
    	\label{tab2}
    	\resizebox{\textwidth}{!}{
    		\begin{tabular}{c|ccccccc}
    			\hline
    			\diagbox{$n$}{$\rho(F_k)$}{$k$} & 1 & 2 & 3 & 4 & 5 & 6 & 7 \\ \hline
    			12 &  3.0751 & \textbf{3.0649}  & 3.2427 &  3.5223 &  \textbackslash{} & \textbackslash{} & \textbackslash{} \\
    			13 & 3.2229 &  \textbf{3.1791}  &3.2951  & 3.5443& 3.8231& \textbackslash{} & \textbackslash{} \\
    			14 & 3.3668 & \textbf{3.3013}  &3.3616  &3.5722& 3.8368& \textbackslash{} & \textbackslash{} \\
    			15 & 3.5064&\textbf{3.4274} &3.4422 & 3.6076& 3.8535  &4.1131& \textbackslash{} \\
    			16 & 3.6418  & 3.5544& \textbf{3.5353}& 3.6526 & 3.8742& 4.1243 & \textbackslash{} \\
    			17 & 3.7731 & 3.6807 & \textbf{3.6377} & 3.7088& 3.8998& 4.1376  & 4.3813 \\
    			18 &3.9006 & 3.8053& \textbf{3.7459} & 3.7767& 3.9318 &4.1536& 4.3906\\ \hline
    			
    	\end{tabular}}
    	
    \end{table}

	For $n \geq 19$, note that $F_{1}=G_{1}\circ (1,1,1,1,n-4), F_{2}=G_{1}\circ (1,1,1,2,n-5)$, let \bm{$x$} be the principal eigenvector of $F_{2}$ and $x_{i}$ correspond to vertices in $V_{i}$ for $i=1,2,3,4,5.$ By $\rho(F_{2})\bm{x}=A(F_{2})\bm{x}$, we have
	\begin{align}
	\rho(F_{2})x_{1}&=x_{2}+x_{3}+2x_{4}, \label{eq3.1}\\
	\rho(F_{2})x_{2}&=x_{1}+2x_{4}, \label{eq3.2}\\
	\rho(F_{2})x_{3}&=x_{1}+(n-5)x_{5}, \label{eq3.3}\\
	\rho(F_{2})x_{4}&=x_{1}+x_{2}, \label{eq3.4}\\
	\rho(F_{2})x_{5}&=x_{3}, \label{eq3.5}
	\end{align}

	From $(\ref{eq3.1})$-$(\ref{eq3.3})$, we have
	\begin{equation*}
	\rho(F_{2})(x_{3}-x_{1}-x_{2})=x_{1}+(n-5)x_{5}-x_{2}-x_{3}-2x_{4}-x_{1}-2x_{4},
	\end{equation*}
    multiplying $\rho(F_{2})$ on both sides, by ($\ref{eq3.4}$) and ($\ref{eq3.5}$), yields
	\begin{equation*}
    \rho(F_{2})^{2}(x_{3}-x_{1}-x_{2})=(n-5)x_{3}-\rho(F_{2})x_{3}-\rho(F_{2})x_{2}-4(x_{1}+x_{2}),
	\end{equation*}
	then
	\begin{equation} \label{eq3.6}
    (\rho(F_{2})^{2}-\rho(F_{2})-4)(x_{3}-x_{1}-x_{2})=(n-9-2\rho(F_{2}))x_{3}+\rho(F_{2})x_{1}.
	\end{equation}

	Since $n \geq 19$ and $K_{n-5,1}$ is a proper subgraph of $F_{2}$, we have $\rho(F_{2})>\rho(K_{n-5,1})=\sqrt{n-5}>3$, thus $\rho(F_{2})^{2}-\rho(F_{2})-4>0$. By Theorem \ref{3.8} and $n\geq 19$, we obtain $\rho(F_{2})<\sqrt{2m(F_{2})-n+1}=\sqrt{2(n+1)-n+1}=\sqrt{n+3}<(n-9)/2$, therefore $n-9-2\rho(F_{2})>0$. Since \bm{$x$} is the principal eigenvector of $F_{2}$, we have $x_{i}>0$.
	
	Thus, it follows from (\ref{eq3.6}) that $x_{3}-x_{1}-x_{2}>0$.
	
	Now we have
	\begin{equation}
	\begin{split}
	\rho(F_{1})-\rho(F_{2}) &\geq \bm{x}^{T}A(F_{2})\bm{x}-\bm{x}^{T}A(F_{1})\bm{x} \\
	&=2x_{4}x_{3}-2x_{4}(x_{1}+x_{2}) \\
	&=2x_{4}(x_{3}-x_{1}-x_{2})>0.
	\end{split}\notag
	\end{equation}
	
	Therefore, $\rho(F_{1})>\rho(F_{2})$, which means $k\geq 2$.
	
\end{proof}

Now we prove that $\rho(G_{10} \circ (1,1,1,1,k-1,n-k-3))<\rho(G_{1} \circ(1, 1, 1,k,n-k-3))$ for $k\geq 2$ and $n\geq 12$ by using a well-known operation.

\begin{theorem}\cite{08} \label{3.10}
	Let $v_{1}, v_{2}$ be two vertices of a connected graph $G$ and let $\{u_{1}, u_{2}, \dots, u_{t}\} \subseteq N(v_{1})\setminus N(v_{2})$. Let $G'$ be the graph obtained from $G$ by rotating the edge $v_{1}u_{i}$ to $v_{2}u_{i}$ for $i=1,2,\dots,t$. If $x_{v_{1}} \leq x_{v_{2}}$, where \textbf{x} is the principal eigenvector of $G$, then $\rho(G')>\rho(G)$.
\end{theorem}

\begin{lemma}\label{3.11}
	For $k\geq 2$ and $n\geq 12$, we have $\rho(G_{10} \circ (1,1,1,1,k-1,n-k-3))<\rho(G_{1} \circ(1, 1, 1,k,n-k-3))$.
\end{lemma}

\begin{proof}
	Let $F_{1}=G_{1} \circ(1, 1, 1,k,n-k-3)$ and $F_{2}=G_{10} \circ (1,1,1,1,k-1,n-k-3)$. Let \bm{$x$} be the principal eigenvector of $F_{2}$ and $x_{i}$ correspond to vertices in $V_{i}$ for $i=1,2,3,4,5,6$.
	
	Let us first suppose that $x_{3} \geq x_{1}$, then by Theorem \ref{3.10} we have $\rho(F_{2})<\rho(F')$, where $F'$ is obtained from $F_{2}$ by rotating the edge $v_{1}v_{4}$ to $v_{3}v_{4}$. Since $F' \cong F_{1}$, we obtain $\rho(F_{2})<\rho(F_{1})$.
	
	Now, suppose that $x_{3} < x_{1}$. Since $F'' \cong F_{1}$, where $F''$ is obtained from $F_{2}$ by rotating the edge $v_{3}v_{5}^{i}$ to $v_{1}v_{5}^{i}$ for $i=1,2,\dots,k-1$, we have $\rho(F_{2})<\rho(F'')=\rho(F_{1})$.
		
	Thus, we complete the proof of the Lemma.
			
\end{proof}

Now we know that the extremal graph of order $n$ and rank $5$ with minimum spectral radius is $G_{10} \circ (1,1,1,1,k,n-4-k)$ for some integer $k$ with $1\leq k\leq \frac{n-4}{2}$ when $n\geq 12$.

For convenience, we set $F_n(i)=G_{10} \circ (1,1,1,1,i,n-4-i)$ and $\mathcal{F}=\{F_n(i):1\leq i\leq \frac{n-4}{2}\}$. It is only remained to find the extremal graph with minimum spectral radius in $\mathcal{F}$.

\begin{theorem} \cite{10} \label{3.12}
	Let $A$ be an $n \times n$ nonnegative matrix. Then the largest eigenvalue $\rho(A) \geq \bm{x}^{T}A\bm{x}$ for any unit vector $\bm{x}$, with equality if and only if $A\bm{x}=\rho(A)\bm{x}$.
\end{theorem}

\begin{lemma}\label{3.13}
Let $\alpha=\frac{6n-37-\sqrt{24n+1}}{18}$ and $n\geq 12$. Then for $1\leq i\leq \frac{n-4}{2}$, we have
$$\rho(F_n(i))>\min\{\rho(F_n(\lfloor \alpha \rfloor)), \rho(F_n(\lceil \alpha \rceil))\}$$
unless $i=\lfloor\alpha \rfloor$ or $\lceil \alpha \rceil$.
\end{lemma}

\begin{proof}
Let $\rho_i=\rho(F_n(i))$. Our aim is to prove that $\rho_i<\rho_{i-1}$ if $2\leq i\leq \lfloor \alpha\rfloor$ and $\rho_i<\rho_{i+1}$ if $\lceil \alpha\rceil\leq i\leq \frac{n-6}{2}$.

Let $\bm{x}_i$ be the principal eigenvector of $F_n(i)$ and $x_{j}^{i}$ correspond to vertices in $V_{j}$ for $j=1,2,3,4,5,6$. Then by $\rho_i\bm{x}_i=A(F_n(i))\bm{x}_i$ we have
\begin{align}
\rho_ix_{1}^{i}&=x_{2}^{i}+x_{3}^{i}+x_{4}^{i}, \label{eq3.7}\\
\rho_ix_{2}^{i}&=x_{1}^{i}+x_{3}^{i}+ix_{5}^{i}, \label{eq3.8}\\
\rho_ix_{3}^{i}&=x_{1}^{i}+x_{2}^{i}+ix_{5}^{i}, \label{eq3.9}\\
\rho_ix_{4}^{i}&=x_{1}^{i}+(n-i-4)x_{6}^{i}, \label{eq3.10}\\
\rho_ix_{5}^{i}&=x_{2}^{i}+x_{3}^{i}, \label{eq3.11}\\
\rho_ix_{6}^{i}&=x_{4}^{i}, \label{eq3.12}
\end{align}

By $(\ref{eq3.8})$ and $(\ref{eq3.9})$, we have
\begin{align*}
&\rho_i(x_{2}^{i}-x_{3}^{i})=x_{3}^{i}-x_{2}^{i} \text{, i.e.,}\\
&(\rho_i+1)(x_{2}^{i}-x_{3}^{i})=0,
\end{align*}
which implies that
\begin{align}\label{eq3.13}
x_{2}^{i}=x_{3}^{i}.
\end{align}

Therefore, by (\ref{eq3.7}) and (\ref{eq3.10})-(\ref{eq3.13}), we have
\begin{align*}
x_{5}^{i}=\frac{2x_{2}^{i}}{\rho_i}&=x_{1}^{i}-\frac{x_{4}^{i}}{\rho_i} \\
&=x_{1}^{i}-x_{6}^{i}\\
&=\rho_ix_{4}^{i}-(n-i-4)x_{6}^{i}-x_{6}^{i}\\
&=\rho_i^{2}x_{6}^{i}-(n-i-4)x_{6}^{i}-x_{6}^{i}\\
&=(\rho_i^{2}-n+i+3)x_{6}^{i},
\end{align*}
and from (\ref{eq3.7})-(\ref{eq3.8}) and (\ref{eq3.11})-(\ref{eq3.13}), we have
\begin{align*}
x_{6}^{i}=\frac{x_{4}^{i}}{\rho_i}&=x_{1}^{i}-\frac{2x_{2}^{i}}{\rho_i} \\
&=x_{1}^{i}-x_{5}^{i}\\
&=(\rho_i-1)x_{2}^{i}-ix_{5}^{i}-x_{5}^{i}\\
&=\frac{\rho_i(\rho_i-1)}{2}x_{5}^{i}-ix_{5}^{i}-x_{5}^{i}\\
&=\frac{1}{2}(\rho_i^{2}-\rho_i-2i-2)x_{5}^{i}.
\end{align*}

Hence, we obtain that
\begin{align}\label{eq3.14}
\begin{cases}
(\rho_i^{2}-n+i+3)(\rho_i^{2}-\rho_i-2i-2)=2,\\
\rho_i^{2}-n+i+3>0,\\
\rho_i^{2}-\rho_i-2i-2>0.\\
\end{cases}
\end{align}

Note that, if we let
\begin{align*}
\begin{cases}
\rho_i^{2}-n+i+3=1,\\
\rho_i^{2}-\rho_i-2i-2=2,
\end{cases}
\end{align*}
then we have
\begin{align*}
\begin{cases}
\rho_i=\sqrt{n-i-2},\\
\rho_i=\frac{1+\sqrt{8i+17}}{2}.
\end{cases}
\end{align*}
By calculation, we can find that $i=\alpha=(6n-37-\sqrt{24n+1})/18$ is the only solution of $\sqrt{n-i-2}=(1+\sqrt{8i+17})/2$. Since $i \in \mathbb{N}$, we will complete the proof by classifying the value of $i$.

\textbf{Case 1.} If $2\leq i\leq \lfloor \alpha\rfloor$.

We have $\sqrt{n-i-2} \geq (1+\sqrt{8i+17})/2$. We claim that $(1+\sqrt{8i+17})/2 \leq \rho_{i}\leq \sqrt{n-i-2}$. Suopose that $\rho_{i}<(1+\sqrt{8i+17})/2$. By (\ref{eq3.14}), we have $0<\rho_i^{2}-n+i+3<1$ and $0<\rho_i^{2}-\rho_i-2i-2<2$. Then $(\rho_i^{2}-n+i+3)(\rho_i^{2}-\rho_i-2i-2)<2$, a contradiction. Suopose that $\rho_{i}>\sqrt{n-i-2}$. By (\ref{eq3.14}), we obtain that $\rho_i^{2}-n+i+3>1$ and $\rho_i^{2}-\rho_i-2i-2>2$. Then $(\rho_i^{2}-n+i+3)(\rho_i^{2}-\rho_i-2i-2)>2$, a contradiction.

Thus we have $(1+\sqrt{8i+17})/2 \leq \rho_{i}\leq \sqrt{n-i-2}$. This induces that $\rho_i^{2}-n+i+3\leq 1$ and $\rho_i^{2}-\rho_i-2i-2\geq 2$, which lead to $x_{6}^{i}\geq x_{5}^{i}$. Therefore
\begin{equation}\label{eq3.15}
\begin{split}
&\rho_{i-1}-\rho_{i}\\
\geq & \bm{x_{i}}^{T}A(F_n(i-1))\bm{x_{i}}-\bm{x_{i}}^{T}A(F_n(i))\bm{x_{i}}\\
= &2x_{5}^{i}(x_{4}^{i}-x_{2}^{i}-x_{3}^{i})\\
= &2\rho_{i}x_{5}^{i}(x_{6}^{i}-x_{5}^{i})\\
\geq&0.
\end{split}
\end{equation}

Now we only need to prove $\rho_{i-1}\neq \rho_{i}$. Suppose that $\rho_{i-1}=\rho_{i}$, then $\rho_{i-1}=\bm{x_{i}}^{T}A(F_n(i-1))\bm{x_{i}}$. By Theorem \ref{3.12}, we have
\begin{align*}
\rho_{i-1}x_{4}^{i}=x_{1}^{i}+(n-i-4)x_{6}^{i}+x_{5}^{i},
\end{align*}
and since
\begin{align*}
\rho_{i}x_{4}^{i}=x_{1}^{i}+(n-i-4)x_{6}^{i},
\end{align*}
we obtain $0=(\rho_{i-1} -\rho_{i})x_{4}^{i}=x_{5}^{i}$, which contradicts to the definition of the principal eigenvector.

Therefore, from (\ref{eq3.15}) we have $\rho_{i-1}>\rho_{i}$ for $2\leq i\leq \lfloor \alpha\rfloor$.

\textbf{Case 2.} If $\lceil \alpha\rceil\leq i\leq \frac{n-6}{2}$.

We have $\sqrt{n-i-2} \leq (1+\sqrt{8i+17})/2$. Similarly, by (\ref{eq3.14}), we conclude that $ \sqrt{n-i-2}\leq \rho_{i}\leq (1+\sqrt{8i+17})/2$. This induces that $\rho_i^{2}-n+i+3\geq 1$ and $\rho_i^{2}-\rho_i-2i-2\leq 2$, which lead to $x_{5}^{i}\geq x_{6}^{i}$, therefore
\begin{equation}\label{eq3.16}
\begin{split}
&\rho_{i+1}-\rho_{i}\\
\geq & \bm{x_{i}}^{T}A(F_n(i+1))\bm{x_{i}}-\bm{x_{i}}^{T}A(F_n(i))\bm{x_{i}}\\
=&2x_{6}^{i}(x_{2}^{i}+x_{3}^{i}-x_{4}^{i})\\
= &2\rho_{i}x_{6}^{i}(x_{5}^{i}-x_{6}^{i})\\
\geq&0.
\end{split}
\end{equation}

Similarly, we have $\rho_{i+1}\neq \rho_{i}$. Using this, from (\ref{eq3.16}), we obtain $\rho_{i+1}> \rho_{i}$ for $\lceil \alpha\rceil \leq i\leq \frac{n-6}{2}$.

Therefore, the proof of Lemma is completed.

\end{proof}

\subsection{Step 4}

It only remains for the case that $5\leq n \leq 11$. Applying Proposition \ref{g1}, \ref{g10} and \ref{g7}, we obtain the extremal graphs with minimum spectral radius in $M_{n}(G_{1})$, $M_{n}(G_{7})$ and $M_{n}(G_{10})$, respectively. And then calculate their spectral radii by using MATLAB, as shown in Table \ref{tab3}, where the extremal graphs and the minimum spectral radii are bolded.
	 \begin{table}[H]\Huge
	\centering
	\caption{The extremal graph with minimum spectral radius in $M_{n}(G_{1}),M_{n}(G_{7}),M_{n}(G_{10}).$}
	\label{tab3}
	\resizebox{\textwidth}{!}{
   \begin{tabular}{ccccccc}
	\hline
	\multirow{2}{*}{n} &
	\multicolumn{2}{c}{$M_{n}(G_{1})$} &
	\multicolumn{2}{c}{$M_{n}(G_{7})$} &
	\multicolumn{2}{c}{$M_{n}(G_{10})$} \\
	\cline{2-7}
	& Extremal graph & Spectral radius& Extremal graph & Spectral radius& Extremal graph & Spectral radius\\ \hline
	5 & $G_{1} \circ(1,1,1,1,1)$ & 2.2143 & \bm{$G_{7} \circ(1,1,1,1,1)$} & \textbf{2.0000} & $\backslash$ & $\backslash$ \\
	6 & \bm{$G_{1} \circ(1,1,1,1,2)$} & \textbf{2.2784} & $G_{7} \circ(2,1,1,1,1)$ & 2.3912 & $G_{10} \circ(1,1,1,1,1,1)$ & 2.6544 \\
	7 & \bm{$G_{1} \circ(1,1,1,1,3)$} & \textbf{2.3686} & $G_{7} \circ(2,1,2,1,1)$ & 2.6813 & $G_{10} \circ(1,1,1,1,1,2)$ & 2.6751 \\
	8 & \bm{$G_{1} \circ(1,1,1,1,4)$} & \textbf{2.4860} & $G_{7} \circ(3,1,2,1,1)$ & 2.9764 & $G_{10} \circ(1,1,1,1,1,3)$ & 2.7033 \\
	9 & \bm{$G_{1} \circ(1,1,1,1,5)$} & \textbf{2.6239} & $G_{7} \circ(3,1,3,1,1)$ & 3.2176 & $G_{10} \circ(1,1,1,1,1,4)$ & 2.7448 \\
	10 & \bm{$G_{1} \circ(1,1,1,1,6)$} & \textbf{2.7724} & $G_{7} \circ(4,1,3,1,1)$ & 3.4630 & $G_{10} \circ(1,1,1,1,1,5)$ & 2.8060 \\
	11 & $G_{1} \circ(1,1,1,1,7)$ & 2.9243 & $G_{7} \circ(4,1,4,1,1)$ & 3.6737 & \bm{$G_{10} \circ(1,1,1,1,1,6)$} & \textbf{2.8915} \\ \hline
	\end{tabular}}
\end{table}

By Table \ref{tab4}, we obtain that when $5\leq n \leq 11$, the extremal graph with minimum spectral radius of rank $5$ is:
\begin{itemize}
	\item $G_{7}=C_{5}$, for $n=5$;
	\item $G_{1} \circ(1,1,1,1,n-4)$, for $6\leq n \leq 10$;
	\item $G_{10} \circ(1,1,1,1,1,n-5)$, for $n=11$.
\end{itemize}

\section{Concluding remarks}

In the last case of Theorem \ref{2}, we obtain that $k \in \{ \lfloor \frac{6n-37-\sqrt{24n+1}}{18} \rfloor$ $,\lceil \frac{6n-37-\sqrt{24n+1}}{18} \rceil\}$. When $12\leq n \leq 23$, we use the MATLAB software to calculate the spectral radii of the graphs in $\mathcal{F}=\{F_n(i):1\leq i\leq \frac{n-4}{2}\}$, as shown in the Table \ref{tab4}, where the minimum spectral radius is bolded. It demonstrates that $k=\lfloor \frac{6n-37-\sqrt{24n+1}}{18} \rfloor$ or $\lceil \frac{6n-37-\sqrt{24n+1}}{18} \rceil$ depends on $n$.

\begin{table}[H]
	\centering
	\caption{$\rho(F_{n}(i))$.}
	\label{tab4}
	\resizebox{\textwidth}{!}{
		\begin{tabular}{c|ccccccccc|c}
			\hline
			\diagbox{$n$}{$\rho(F_{n}(i))$}{$i$} & 1 & 2 & 3 & 4 & 5 & 6 & 7 & 8 & 9& $\frac{6n-37-\sqrt{24n+1}}{18}$\\ \hline
			12 & \textbf{3} & 3.1370 & 3.4319& 3.7362 &  \textbackslash{} & \textbackslash{} & \textbackslash{} & \textbackslash{} & \textbackslash{}&1 \\
			13 & \textbf{3.1239} & 3.1818 & 3.4431& 3.7404& \textbackslash{} & \textbackslash{} & \textbackslash{} & \textbackslash{} & \textbackslash{}& 1.2949\\
			14 & 3.255 &\textbf{3.2470}&3.4588&3.7457&4.0278& \textbackslash{} & \textbackslash{} & \textbackslash{} & \textbackslash{} & 1.5912\\
			15 & 3.3894& \textbf{3.3347} & 3.4817 & 3.7525& 4.0308& \textbackslash{} & \textbackslash{} & \textbackslash{} & \textbackslash{} & 1.8889 \\
			16 & 3.5227& \textbf{3.4402}& 3.5160& 3.7616& 4.0344& 4.2979 & \textbackslash{} & \textbackslash{} & \textbackslash{} & 2.1877  \\
			17 & 3.6539 & \textbf{3.5563} & 3.5674 & 3.7743& 4.0389& 4.3001 & \textbackslash{} & \textbackslash{} & \textbackslash{} & 2.4876 \\
			18 &3.7824 & 3.6770& \textbf{3.6394} & 3.7926& 4.0446 &4.3027& 4.5506& \textbackslash{} & \textbackslash{} & 2.7884 \\
			19   & 3.9079& 3.7889& \textbf{3.7303}&3.8199& 4.0523 & 4.3058 &4.5522& \textbackslash{} & \textbackslash{} & 3.0901 \\
			20   & 4.0303 & 3.9201& \textbf{3.8338} & 3.8612& 4.0628 & 4.3097& 4.5542& 4.7888& \textbackslash{} & 3.3927 \\
			21   &4.1498&4.0396& 3.9439&\textbf{3.9211}& 4.0779& 4.3147 & 4.5565 & 4.7900& \textbackslash{} & 3.6960\\
			22   &4.2663  & 4.1570 & 4.0564 & \textbf{4} & 4.1002 &  4.3213& 4.5593 & 4.7915& 5.0146 & 4\\
			23   &4.3801 & 4.2721   & 4.1694 & \textbf{4.0929 } & 4.1341 &  4.3303 & 4.5627  &  4.7933& 5.0157 & 4.3047\\ \hline
	\end{tabular}}
	
\end{table}

	It is a natural problem to determine the extramal spectral radii of the  graphs of order $n$ and rank $r$. By Theorem \ref{1}, we know that the maximum spectral radius of all connected graphs of order $n$ and rank $r$ is $\rho(T(n,r))$. Feng et al. gave the spectral radius of $T(n,r)$ in \cite{13}.
	\begin{theorem}\cite{13}
		Let $T(n,r)$ be a Tur\'{a}n graph. Then
		$$\rho(T(n,r))=\frac{1}{2}\left(n-2\lfloor \frac{n}{r} \rfloor-1+\sqrt{(n+1)^{2}-4(n-r\lfloor \frac{n}{r} \rfloor)\lceil \frac{n}{r} \rceil}\right)\leq n-\lfloor \frac{n}{r} \rfloor$$
		with the last equality if and only if $T(n,r)$ is regular.
	\end{theorem}
	
	Further, we obtain a sharp upper and lower bound for the spectral radius of the extremal graph $G$ which attains the minimum spectral radius among all connected graphs of order $n\geq 12$ and rank $5$. By Theorem \ref{2}, we know that
	\begin{align*}
	\rho(G)= \ \text{min} \ \{\rho(F_n(\lfloor \alpha\rfloor)), \rho(F_n(\lceil \alpha\rceil)) \},
	\end{align*}
	where $\alpha=\frac{6n-37-\sqrt{24n+1}}{18}$.

	From the proof of Lemma \ref{3.13}, we have
	\begin{align*}
	\frac{1+\sqrt{8\lfloor \alpha\rfloor +17}}{2} \leq \rho(F_n(\lfloor \alpha\rfloor) \leq \sqrt{n-\lfloor \alpha\rfloor-2},\\
	\sqrt{n-\lceil \alpha\rceil-2} \leq \rho(F_n(\lceil \alpha\rceil) \leq \frac{1+\sqrt{8\lceil \alpha \rceil +17}}{2}.
	\end{align*}
	Therefore, we obtain that
	\begin{align*}
	 \rho(G) \geq \text{min} \{ \frac{1+\sqrt{8\lfloor \alpha\rfloor +17}}{2}, \sqrt{n-\lceil \alpha\rceil-2}\},
	\end{align*}
	and
	\begin{align*}
	\rho(G) \leq \text{min} \{\sqrt{n-\lfloor \alpha\rfloor-2},  \frac{1+\sqrt{8\lceil \alpha \rceil +17}}{2}\}.
	\end{align*}
In general, the problem of determining the minimum spectral radius of all connected graphs with order $n$ and rank $r$ deserves further study.

\section*{Declaration of compting interest}

There is no competing interest.

\section*{Acknowledgement}

This research is supported by the National Natural Science Foundation of China [Grant number, 12171402].

\end{document}